\documentclass[12pt]{article}
\renewcommand{\baselinestretch}{1.5}

\usepackage{amsmath,amssymb,amsthm,dsfont,mathtools,bm}
\usepackage[shortlabels]{enumitem}
\usepackage[round]{natbib}

\usepackage{multirow,array,booktabs}
    \newcommand{\otoprule}{\midrule[\heavyrulewidth]}
    \newcolumntype{Q}{>{$}l<{$}}

\newcommand*{\abs}[1]{\left\lvert#1\right\rvert}
\newcommand*{\set}[1]{\left\{#1\right\}}
\newcommand{\prob}{\mathbf{P}}
\newcommand{\E}{\mathbf{E}}
\newcommand{\ind}{\mathds 1}
\newcommand{\sd}{\bm{\sigma}}

\newtheorem{thm}{Theorem}
\newtheorem{lem}[thm]{Lemma}
\newtheorem{prop}[thm]{Proposition}
\newtheorem{cor}[thm]{Corollary}
\theoremstyle{definition}

\theoremstyle{remark}
\newtheorem{remark}{Remark}

\allowdisplaybreaks

\begin{document}

\noindent TWO METHODS OF ESTIMATION OF THE DRIFT PARAMETERS\\ OF THE COX--INGERSOLL--ROSS PROCESS: CONTINUOUS OBSERVATIONS 
\vskip 3mm

\vskip 5mm
\noindent Olena Dehtiar, Yuliya Mishura and Kostiantyn Ralchenko

\noindent Department of Probability, Statistics and Actuarial Mathematics

\noindent Mechanics and Mathematics Faculty

\noindent Taras Shevchenko National University of Kyiv

\noindent 64 Volodymyrska, 01601 Kiev, Ukraine

\noindent \ 

\vskip 3mm
\noindent Key Words: Cox--Ingersoll--Ross process, continuous observations, parameter estimation, strong
consistency.
\vskip 3mm

\noindent ABSTRACT

    We consider a stochastic differential equation of the form
	$dr_t = (a - b r_t) dt + \sigma\sqrt{r_t}dW_t$,
	where $a$, $b$ and $\sigma$ are positive constants.
	The solution corresponds to the Cox--Ingersoll--Ross process.
	We study the estimation of an unknown drift parameter $(a,b)$ by continuous observations of a sample path $\set{r_t,t\in[0,T]}$.
	First, we prove the strong consistency of the maximum likelihood estimator. Since this estimator is well-defined only in the case $2a>\sigma^2$, we propose another estimator that is defined and strongly consistent for all positive $a$, $b$, $\sigma$.
	The quality of the estimators is illustrated by simulation results.
\vskip 4mm

\noindent 1.   INTRODUCTION

	 The Cox--Ingersoll--Ross (CIR) process  is a very famous object and  is a unique solution of the following stochastic differential equation
	\begin{equation}\label{equ:1}
	dr_t = (a - b r_t) dt + \sigma\sqrt{r_t}dW_t,\quad r_t\big|_{t=0}=r_0>0
	\end{equation}
	where $W=\{W_t, t \ge 0\}$ is a Wiener process and $a,b,\sigma$ are positive constants.
	
	There are many papers devoted to the construction and the asymptotic properties of the parameter estimators of the CIR process, based on conditional least squares and the maximum likelihood estimator  see,  e.\,g., \cite{Overbeck1997,Li2015,Rosi2010,Barczy2017,BenAlaya2012,BenAlaya2013}. More precisely,  \cite{Overbeck1997}  used conditional least squares as a basic method  and derived two estimators, which differ by the method of estimating of parameter   $\sigma$: pseudo likelihood method and unweighted least squares method based on squared residuals. Moreover, strong consistency of these estimators was proved for both observations at equidistant time point and continuous observations.  Nevertheless, the simulation study demonstrated  that the maximum likelihood estimator outperforms the conditional least-squares estimators and the pseudo likelihood approach in most cases.

    \cite{BenAlaya2012,BenAlaya2013} presented a new approach to the investigation of the asymptotic behavior of the maximum likelihood estimator, depending on the values of the parameters, both for continuous and discrete observations. Using an exact simulation algorithm, the authors illustrated practical behavior of these estimators' errors covering ergodic and nonergodic situations.
    Asymptotic properties of maximum likelihood estimator for the stable CIR process were studied in \cite{Li2015}. In \cite{Barczy2017} the authors investigated the asymptotic behavior of the maximum likelihood estimator for so-called jump-type CIR process, driven by a standard Wiener process and a subordinator. They distinguish three cases: subcritical, critical and supercritical.

    The paper  \cite{Rosi2010} contains the  method that is based on the sequential Monte Carlo techniques and shows how to construct a simulated maximum  likelihood procedure. This paper also describes two methods of computing the likelihood: sampling and re-sampling algorithm which solve the problem of degeneracy for realistic sample sizes. In order to maximize the likelihood, the author applies the genetic algorithm that relies on the survival of the fittest in determining the optimal parameter vector. Testing on simulated data confirmed that this approach allows not to undermine the accuracy of the estimation procedure by the effect of simulation errors and copes with larger parameter dimensions at a modest computational cost.
    Avoiding  the  computational  burden  the  MATLAB implementation  of  the  estimation  routine  is  provided  and  tested  in  \cite{Kladivko2007}. Moreover, the simulation algorithm for the approximation of the CIR process trajectories was described in \cite{Milstein2013}.

    Concerning the discretization, due to the square-root diffusion coefficient the classical Euler--Maruyama scheme does not preserve the non-negativity of the process. \cite{Cozma2016}  provided a brief discussion of the discretization schemes often encountered in the finance literature. They also investigated the exponential integrability properties which play an essential role in deriving strong convergence of Euler discretization schemes. The high priority to the convergence results and its applications were given in the paper \cite{Deelstra1995}.
    The Euler approach to constructing discretization schemes was also introduced in \cite{Dereich2012}.  Here the using of  the  drift-implicit  square-root  Euler  method  gives  a  strictly  positive  approximation  of  the original  CIR  process  and  some  global  convergence  results. The  paper  \cite{Mishura2016}  is  the  example of  constructing  additive  and  multiplicative  discrete  approximation  schemes  taking  the  Euler approximations of the CIR process itself but replacing the increments of the Wiener process with i.\,i.\,d.\ bounded vanishing symmetric random variables.

 In this paper we investigate two estimators of the parameter $(a,b)$ by continuous observations of a sample path of CIR process $r=\set{r_t,t\in[0,T]}$ and prove their strong consistency as $T\to\infty$.
    The first one is the standard maximum likelihood estimator, which was constructed and studied in \cite{BenAlaya2012,BenAlaya2013}.
    Compared to the known results, we establish the strong consistency instead of weak one.
    Note that the maximum likelihood estimator is well-defined only if $2a\ge\sigma^2$, because, in particular,  it contains the integral $\int_0^T\frac1{r_t}dt$, which exists with probability one if and only if $2a\ge\sigma^2$, see \cite[Prop.~4]{BenAlaya2012}. For this reason, we decided to create some statistics that converges regardless of whether $2a\ge\sigma^2$ or not. On this way  we created   a different estimator of the vector parameter $(a,b)$, which is strongly consistent for all positive $a$, $\sigma$ and $b$.
    Another advantage of the new alternative  estimator is that it has simpler form, therefore, it is computationally faster. It includes only two statistics of the process $r$, namely the Lebesgue integrals $\int_0^Tr_t\,dt$ and $\int_0^Tr_t^2\,dt$, see Theorem~\ref{th:consist-new} below.
At the same time the maximum likelihood estimator depends on two Lebesgue integrals, $\int_0^Tr_t\,dt$ and $\int_{0}^{T}\frac{dt}{r_t}$, on the stochastic integral $\int_{0}^{T}\frac{dr_t}{r_t}$ and on the process itself.
    Particular attention in the paper is paid to the a.\,s. asymptotic behavior of the integral
    $\int_0^T r_t^2 dt$, which is crucial for the construction of the alternative estimator.
    We would like to emphasize that in the present paper we restrict ourselves to the case $b>0$. The boundary case $b=0$ was investigated in \cite{BenAlaya2013}, where the consistency and asymptotic distribution on the maximum likelihood estimator was derived, assuming that $2a\ge\sigma^2$.
    To the best of our knowledge, parameter estimation for the case $b<0$ remains an open problem.

The paper is organized as follows.
    In Section~2 we recall some basic properties of CIR process.
    Section~3 is devoted to the strong consistency of the maximum likelihood estimator.
    In Section~4 we introduce an alternative estimator and prove its strong consistency.
\vskip 3mm
    	
\noindent 2.    PRELIMINARIES
  
    In this section we consider the properties of the CIR process. Most of them will  be useful for statistical parameter estimation, but some facts are of independent interest. These facts are well-known, but we combined them into one statement for the reader's convenience.
	\begin{prop}\label{prop:CIR}
	Assume that $2a> \sigma^2$.	
Then
	\begin{enumerate}[1)]
			\item The unique solution $r=\{r_t, t\ge 0\}$ of equation \eqref{equ:1} is positive with probability 1:
\begin{equation}\label{positivity}
\inf\set{t\ge0:r_t=0}=+\infty \quad \text{a.\,s.}
\end{equation}		
(with convention $\inf\emptyset=+\infty$).	
Moreover,
\begin{equation}\label{limits}
\prob\set{\limsup_{t\rightarrow\infty} r_t=+\infty}= \prob\set{\liminf_{t\rightarrow\infty} r_t=0}=1.
\end{equation}
			\item The process $r$ is ergodic and it has continuous stationary density  that corresponds to gamma distribution and has the following form:
\begin{equation}\label{gamma}
p_{\infty}(x)=\frac{\beta^\alpha}{\Gamma(\alpha)}x^{\alpha-1}e^{-\beta x}\ind_{x>0}, \quad \alpha=\frac{2a}{\sigma^2},\;\beta=\frac{2b}{\sigma^2}.
\end{equation}
			\item  For any function $f\colon\mathbb{R}^+ \rightarrow \mathbb{R}$ such that $\int_{\mathbb{R}^+}|f(x)|p_{\infty}(x)dx<\infty$ we have that
			\begin{equation}\label{ergo}
			\frac{1}{T}\int_0^Tf(r_t)dt\rightarrow \int_{\mathbb{R}} f(x)p_\infty(x)dx,\quad\text{a.\,s., as } T\rightarrow\infty.
			\end{equation}
		\end{enumerate}
	\end{prop}
	
    The results of Proposition~\ref{prop:CIR} follow from the general theory of homogeneous diffusions.  In particular, \eqref{positivity} follows from the Feller's test for explosions \citep[see, e.\,g.,][Thm.~5.29, p.~348]{KaratzasShreve91}, and \eqref{limits} can be deduced from \citet[Prop.~5.22, p.~345]{KaratzasShreve91} \citep[see also][Sec.~2.3]{MijatovicUrusov2012}. Statements 2)--3) are based on the ergodic theory for homogeneous diffusions, see, e.\,g., \citet[Thm.~1.16]{Kutoyants} or \citet[Ch. 1, \S~3]{Skorokhod-asymp}. The conditions of these general results for the case of the CIR process can be easily  verified, see, e.\,g., \citet[Example 4, p.~280]{meriem}.  A direct proof of \eqref{positivity} for the CIR process is given in	 \citet[Sec.~1.2.4]{Alfonsi}.
     The derivation of the stationary distribution \eqref{gamma} can be found, e.\,g., in \citet[Eq.~(1.24)]{Alfonsi}.

	\begin{cor}\label{cor1.2} It follows immediately from \eqref{ergo} (see also Remark after Proposition 4 in \cite{BenAlaya2012}) that  in the case  $2a> \sigma^2$ we have the following asymptotic relations:
		\begin{gather}
		\frac{1}{T}\int_0^T r_t\,dt\rightarrow \int_{\mathbb{R}}  xp_\infty(x)\,dx=\frac{\alpha}{\beta}=\frac{a}{b},\quad\text{a.\,s., as } T\rightarrow\infty,
		\label{ergo-1}
		\\
		\frac{1}{T}\int_0^T \frac{1}{r_t}\,dt\rightarrow \int_{\mathbb{R}} \frac{1}{x}p_\infty(x)\,dx=\frac{\beta}{\alpha-1}=\frac{b}{a-\sigma^2/2},\quad\text{a.\,s., as } T\rightarrow\infty,
		\label{ergo-2}
		\\
		\frac{1}{T}\int_0^T r_t^2\,dt \to \int_{\mathbb{R}} x^2 p_\infty(x)dx = \left (\frac{\alpha}{\beta}\right )^2 + \frac{\alpha}{\beta^2}
				=\frac{a^2}{b^2} + \frac{a\sigma^2}{2b^2},
				\quad\text{a.\,s., as } T\rightarrow\infty.
		\label{ergo-3}		
		\end{gather}
	\end{cor}	
	
\begin{remark}\label{rem:conv}
It was proved in \cite[Thm.~1]{Deelstra1995} that the convergence \eqref{ergo-1} holds also in the case $0<2a\le\sigma^2$.
The convergence \eqref{ergo-3} is valid for all positive $a$ and $\sigma$, this will be justified in Theorem~\ref{th:int-r2} below.
\end{remark}	
\vskip 3mm
    	
\noindent 3.  MAXIMUM LIKELIHOOD ESTIMATION

	Let us recall the construction of the maximum likelihood estimator of the couple of unknown parameters $(a,b)$ by the continuous observations of $r$ over the interval $[0,T]$.
	We assume that $2a> \sigma^2$ throughout this section.

Dividing \eqref{equ:1} by $\sqrt{r_t}$ and integrating \eqref{equ:1} over the time interval $[0,s]$ we get  the equality:
		\begin{equation*}
		\int_{0}^{s}\frac{dr_t}{\sqrt{r_t}} = \int_{0}^{s}\left(\frac{a}{\sqrt{r_t}}-b \sqrt{r_t}\right)dt +  \sigma W_s.
		\end{equation*}
		
		In order to construct likelihood function for the estimation of couple $(a,b)$ of parameters, we
		use the  Girsanov theorem for the Wiener process with the drift that equals $$\frac{1}{\sigma}\int_{0}^{s} \left(\frac{a}{\sqrt{r_t}}-b \sqrt{r_t}\right)dt.$$
		Then the likelihood function that corresponds  to the likelihood  $dQ_{(0,0)}/dQ_{(a,b)}$, where $Q_{(0,0)}$ is a probability measure responsible for zero values of parameters, and $Q_{(a,b)}$ is responsible for the couple  $(a,b)$, gets the following form:
		\begin{align*}
		\mathcal{L} & = \exp \left\{ - \int_{0}^{T} \frac{a - b r_t}{\sigma \sqrt{r_t}}dW_t - \frac{1}{2} \int_{0}^{T} \frac{(a - b r_t)^2}{\sigma^2 r_t}dt \right\}\\
		&= \exp \left\{- \int_{0}^{T} \frac{a - b r_t}{\sigma^2 r_t}dr_t + \frac{1}{2} \int_{0}^{T} \frac{(a - b r_t)^2}{\sigma^2 r_t}dt\right\}.
		\end{align*}
		Since we are interested in maximization of the function that corresponds to $dQ_{(a,b)}/dQ_{(0,0)}$, our likelihood function has the form
		\begin{align*}
		\tilde{\mathcal{L}} & =
		\exp \left\{  \int_{0}^{T} \frac{a - b r_t}{\sigma^2 r_t}dr_t - \frac{1}{2} \int_{0}^{T} \frac{(a - b r_t)^2}{\sigma^2 r_t}dt\right\}.
		\end{align*}		
The maximum likelihood estimator for the couple $(a,b)$ is constructed by maximizing of $\tilde{\mathcal{L}}$ with respect to $(a,b)$.
It has the following form:
		\begin{equation}\label{equ:5}
		\hat a_T  =\frac{ \int_{0}^{T}r_tdt \int_{0}^{T}\frac{dr_t}{r_t} - T\cdot (r_T-r_0)}{ \int_{0}^{T}r_t dt \cdot\int_{0}^{T}\frac{dt}{r_t} - T^2};
		\end{equation}
		\begin{equation}\label{equ:6}
		\hat b_T =\frac{(r_0-r_T) \int_{0}^{T} \frac{dt}{r_t} + T \int_{0}^{T}\frac{dr_t}{r_t}}{ \int_{0}^{T}r_t dt \cdot\int_{0}^{T}\frac{dt}{r_t} - T^2}.
		\end{equation}

	\begin{thm}\label{th:consist-mle}
	Assume that $2a>\sigma^2$. Then the estimator $(\hat a_T, \hat b_T)$ is strongly consistent.
	\end{thm}
	\begin{proof}
		Taking into consideration that ${ \int_{0}^{T}\frac{dr_t}{r_t} = a \int_{0}^{T}\frac{dt}{r_t} - b T + \sigma \int_{0}^{T}\frac{dW_t}{\sqrt{r_t}}}$ we can represent \eqref{equ:5} in the following form:
		\begin{align}\label{equ:7}
		\hat a_T &= \frac{ a\int_{0}^{T}\frac{dt}{r_t}\int_{0}^{T}r_t dt - bT\int_{0}^{T}r_t dt + \sigma  \int_{0}^{T}\frac{dW_t}{\sqrt{r_t}}\int_{0}^{T}r_t dt - T(r_T-r_0)}{ \int_{0}^{T}r_t dt \cdot \int_{0}^{T}\frac{dt}{r_t} - T^2} {}\nonumber\\
		& =\frac{a\int_{0}^{T}\frac{dt}{r_t}\int_{0}^{T}r_t dt - bT\int_{0}^{T}r_t dt + \sigma  \int_{0}^{T}\frac{dW_t}{\sqrt{r_t}}\int_{0}^{T}r_t dt -T\int_{0}^{T} (a - b r_t)dt - \sigma T  \int_{0}^{T}\sqrt{r_t}dW_t}{ \int_{0}^{T}r_t dt \cdot \int_{0}^{T}\frac{dt}{r_t} - T^2} {}\nonumber\\
		&= a + \frac{ \sigma  \int_{0}^{T}\frac{dW_t}{\sqrt{r_t}} \cdot \int_{0}^{T}r_t{dt}- \sigma T \int_{0}^{T}\sqrt{r_t}dW_t}{ \int_{0}^{T}r_t dt \cdot \int_{0}^{T}\frac{dt}{r_t} - T^2},
		\end{align}
and similarly, estimator from \eqref{equ:6} can be presented as
		\begin{align}\label{equ:8}
		\hat b_T &= \frac{ (r_0-r_T) \int_{0}^{T} \frac{dt}{r_t} + a T \int_{0}^{T}\frac{dt}{r_t} - b T^2 + \sigma T \int_{0}^{T}\frac{dW_t}{\sqrt{r_t}}}{ \int_{0}^{T}r_t dt \cdot \int_{0}^{T}\frac{dt}{r_t} - T^2} {}\nonumber\\
		& =\frac{ -\int_{0}^{T} (a - b r_t)dt \cdot \int_{0}^{T}\frac{dt}{r_t} - \sigma  \int_{0}^{T}\sqrt{r_t}dW_t\cdot \int_{0}^{T}\frac{dt}{r_t} + a T \int_{0}^{T}\frac{dt}{r_t} - b T^2 +  \sigma T\int_{0}^{T}\frac{dW_t}{\sqrt{r_t}}}{ \int_{0}^{T}r_t dt \cdot \int_{0}^{T}\frac{dt}{r_t} - T^2} \nonumber\\
		&= b + \frac{ -\sigma  \int_{0}^{T}\sqrt{r_t}dW_t \cdot \int_{0}^{T}\frac{dt}{r_t}+ \sigma T \int_{0}^{T}\frac{dW_t}{\sqrt{r_t}}}{ \int_{0}^{T}r_t dt \cdot \int_{0}^{T}\frac{dt}{r_t} - T^2}.
		\end{align}
Consider the remainder $$R_T^a=\frac{ \sigma  \int_{0}^{T}\frac{dW_t}{\sqrt{r_t}} \cdot \int_{0}^{T}r_t{dt}- \sigma T \int_{0}^{T}\sqrt{r_t}dW_t}{ \int_{0}^{T}r_t dt \cdot \int_{0}^{T}\frac{dt}{r_t} - T^2}$$ from \eqref{equ:7}, and the remainder
	$$R_T^b=\frac{ -\sigma  \int_{0}^{T}\sqrt{r_t}dW_t \cdot \int_{0}^{T}\frac{dt}{r_t}+ \sigma T \int_{0}^{T}\frac{dW_t}{\sqrt{r_t}}}{ \int_{0}^{T}r_t dt \cdot \int_{0}^{T}\frac{dt}{r_t} - T^2}$$ from \eqref{equ:8}	can be considered in the same lines.
So,
    \begin{align}\label{equ:9}
		R_T^a=\frac{ \sigma  \left(\frac{1}{T}\int_{0}^{T}\frac{dW_t}{\sqrt{r_t}}\right) \cdot \left(\frac{1}{T}\int_{0}^{T}r_t{dt}\right)- \sigma \left(\frac{1}{T} \int_{0}^{T}\sqrt{r_t}dW_t\right)}{ \frac{1}{T^2}\int_{0}^{T}r_t dt \cdot \int_{0}^{T}\frac{dt}{r_t} - 1}
	\end{align}
According to Corollary~\ref{cor1.2}, relations \eqref{ergo} and \eqref{ergo-1}, the denominator in \eqref{equ:9} tends to $\frac{a}{a-\frac{\sigma^2}{2}}-1 > 0.$
Moreover, both values in the numerator, $\frac{1}{T}\int_{0}^{T}\frac{dW_t}{\sqrt{r_t}}$ and $\frac{1}{T} \int_{0}^{T}\sqrt{r_t}dW_t$ tends to zero a.s. as $T \rightarrow \infty$.
Indeed,
$$\frac{1}{T}\int_{0}^{T}\frac{dW_t}{\sqrt{r_t}} = \frac{\int_{0}^{T}\frac{dt}{r_t}}{T} \cdot \frac{\int_{0}^{T}\frac{dW_t}{\sqrt{r_t}}}{\int_{0}^{T}\frac{dt}{r_t}},$$
and $$\frac{\int_{0}^{T}\frac{dt}{r_t}}{T} \rightarrow \frac{b}{a-\frac{\sigma^2}{2}}.$$
This means that $\int_{0}^{T}\frac{dt}{r_t} \rightarrow \infty, T \rightarrow \infty,$ but $\int_{0}^{T}\frac{dt}{r_t}$ is a square characteristics of the locally square integrable martingale $\int_{0}^{T}\frac{dW_t}{\sqrt{r_t}}.$
According to the strong law of large numbers for locally square integrable martingales (\cite{lipshir}), $\left(\int_{0}^{T}\frac{dt}{r_t}\right)^{-1} \cdot  \int_{0}^{T}\frac{dW_t}{\sqrt{r_t}}\rightarrow 0$ a.s. as $T \rightarrow \infty.$
Therefore, $\frac{1}{T}\int_{0}^{T}\frac{dW_t}{\sqrt{r_t}} \rightarrow 0$ a.s. as $T \rightarrow \infty$, and similar relation holds for $\frac{1}{T}\int_{0}^{T}\sqrt{r_t}dW_t$.
Together with \eqref{equ:9}, this means that $\hat a_T$ is strongly consistent.
	\end{proof}

\begin{remark}
The weak consistency and asymptotic normality of the maximum likelihood estimator $(\hat a,\hat b)$ was obtained in \cite{BenAlaya2012,BenAlaya2013}.
Mention that the weak consistency holds also in the boundary cases, when $2a=\sigma^2$ and/or $b=0$ (however, joint asymptotic distribution is not normal in these cases). If $2a<\sigma^2$, then the maximum likelihood estimator is not well-defined. If $2a\ge\sigma^2$ and $b<0$, then it is not consistent.
\end{remark}
	
	\begin{remark}
An alternative parametrization of the CIR process is
		\begin{equation}\label{equ:10}
		dr_t = \alpha(\mu -  r_t) dt + \sigma\sqrt{r_t}dW_t, \quad r_t\bigr|_{t=0}=r_0>0,
		\end{equation}
where $W=\{W_t, t \ge 0\}$ is a Wiener process and $\alpha,\mu,\sigma$ are positive constants. Assume that $2\alpha\mu > \sigma^2$. Then the maximum likelihood estimator for the couple $(\alpha,\mu)$ constructed by the continuous observations of $r$ over the interval $[0,T]$ has the following form
		\[
		\hat \alpha_T =\frac{ T  \int_{0}^{T} \frac{dr_t}{r_t} - \int_{0}^{T} \frac{dt}{r_t}\int_{0}^{T} dr_t}{ \int_{0}^{T}r_t dt \int_{0}^{T}\frac{dt}{r_t} - T^2},
\qquad
		\hat \mu_T =\frac{\int_{0}^{T}r_t dt \int_{0}^{T}\frac{dr_t}{r_t} - T(r_T-r_0)}{T  \int_{0}^{T} \frac{dr_t}{r_t} - \int_{0}^{T} \frac{dt}{r_t} (r_T-r_0)}.
		\]
The strong consistency of $(\hat \alpha_T, \hat \mu_T)$ follows from Theorem~\ref{th:consist-mle}, if we take into account the relations
\begin{equation}\label{equ:dif-param}
\alpha=b, \quad \mu=\frac{a}{b}, \quad \hat\alpha_T=\hat b_T, \quad \hat\mu_T=\frac{\hat a_T}{\hat b_T}.
\end{equation}
	\end{remark}
\vskip 3mm
    	
\noindent 4.   AN ALTERNATIVE APPROACH TO DRIFT PARAMETERS ESTIMATION

The disadvantage of the maximum likelihood estimators is that they work only if $a>\sigma^2/2$, however, a priori we do not know if this relation holds for the observed process.  To avoid this circumstance, in this section we will introduce a new estimator for the parameter $(a,b)$ based on the statistics $\int_0^T r_t\,dt$ and $\int_0^T r_t^2\,dt$.
First, we will prove that the convergence \eqref{ergo-3} remains valid in the case  $0<a\le\sigma^2/2$.
To this end, we start with several auxiliary results. The first result gives the asymptotics of  two normalized  Lebesgue integrals

\begin{lem}\label{l:conv-leb-int}
Let $a>0$, $b>0$, $\sigma>0$. Then the following normalized integrals asymptotically vanish as  $T\to\infty$:
\begin{gather}
\frac1T\int_0^T e^{-bt} r_t\,dt \to 0
\quad\text{a.\,s.,}
\label{equ:int-e-r}
\\
\frac1T e^{-2b T}\int_0^T e^{2bt} r_t\,dt \to 0
\quad\text{a.\,s.}
\label{equ:e-int-e-r}
\end{gather}
\end{lem}
\begin{proof}
1. In order to prove the asymptotic relation \eqref{equ:int-e-r}, we rewrite the normalized integral as follows:
\[
\frac1T\int_0^T e^{-bt} r_t\,dt
= \frac1T\int_0^{\sqrt T} e^{-bt} r_t\,dt
+ \frac1T\int_{\sqrt T}^T e^{-bt} r_t\,dt.
\]
Then the first integral can be bounded in the following way:
\[
\frac1T\int_0^{\sqrt T} e^{-bt} r_t\,dt
\le \frac1T\int_0^{\sqrt T} r_t\,dt
=\frac{1}{\sqrt T} \cdot \frac{1}{\sqrt T}\int_0^{\sqrt T} r_t\,dt
\to 0
\quad\text{a.\,s., as } T\to\infty,
\]
since $\frac{1}{\sqrt T}\int_0^{\sqrt T} r_t\,dt
\to \frac{a}{b}$ a.\,s., as  $T\to\infty$, by \eqref{ergo-1}, see Remark~\ref{rem:conv}.
Furthermore, the second integral can be bounded as
\[
\frac1T\int_{\sqrt T}^T e^{-bt} r_t\,dt
\le e^{-b\sqrt T} \frac1T\int_{\sqrt T}^T  r_t\,dt
\le e^{-b\sqrt T} \frac1T\int_{0}^T  r_t\,dt
\to 0
\quad\text{a.\,s., as } T\to\infty,
\]
where the convergence follows from \eqref{ergo-1}.
Thus, relation  \eqref{equ:int-e-r} is proved.

2. Note that $\int_0^T e^{2bt} r_t\,dt\ge \int_0^T r_t\,dt \to \infty$ a.\,s., as $T\to\infty$, by \eqref{ergo-1}.
Therefore, applying the  L'H\^opital's rule, we conclude that
\[
\lim_{T\to\infty} \frac{\int_0^T e^{2bt} r_t\,dt}{T e^{2bT}}
=
\lim_{T\to\infty} \frac{e^{2bT} r_T}{e^{2bT} + 2b e^{2bT} T}
=
\lim_{T\to\infty} \frac{r_T}{1 + 2bT} \quad\text{a.\,s.}
\]
Now \eqref{equ:e-int-e-r} follows from the a.\,s.\ convergence $\frac{r_T}{T}\to0$, $T\to\infty$. The latter convergence  was established in the proof of Theorem 1 of   \cite{Deelstra1995}.
\end{proof}

The next result presents the bounds for the moments of the  related    stochastic integral.
\begin{lem}
Denote
\begin{equation}\label{equ:Z}
Z_t = \int_0^t e^{bu}\sqrt{r_u}\,dW_u, \quad t\ge0.
\end{equation}
There exists a constant $C>0$ such that for all $t\ge0$,
\begin{equation}\label{equ:E-bounds}
\E \left[Z_t^2\right] \le C e^{2bt},
\qquad
\E \left[Z_t^3\right] \le C e^{3bt},
\qquad
\E \left[Z^2_t r_t\right] \le C e^{2bt}.
\end{equation}
\end{lem}

\begin{proof}
1. According to \cite[Eq.~(1)]{Deelstra1995}, the process $r$ satisfies the following relations
\begin{align}
r_t &= e^{-bt}\left(r_0 + a\int_0^te^{bu}\,du + \sigma \int_0^t e^{bu} \sqrt{r_u}\,dW_u\right)
\notag\\
&= e^{-bt}\left(r_0 - \frac{a}{b} + \frac{a}{b}e^{bt} + \sigma \int_0^t e^{bu} \sqrt{r_u}\,dW_u\right).
\label{equ:CIR-2}
\end{align}
Therefore, in particular,  its expectation equals
\[
\E r_t = \left(r_0 - \frac{a}{b}\right) e^{-bt} + \frac{a}{b}.
\]
Then for the 2nd moment we have the following representation:
\begin{align*}
\E Z_t^2 &= \E\left[\left(\int_0^t e^{bu}\sqrt{r_u}\,dW_u\right)^2\right]
= \E\left[\int_0^t e^{2bu}r_u\,du\right]
= \int_0^t \left (\left(r_0 - \frac{a}{b}\right) e^{bu} + \frac{a}{b}e^{2bu}\right) du
\\
&= \left(r_0 - \frac{a}{b}\right) \frac{e^{bt}-1}{b}
+ \frac{a}{b}\cdot\frac{e^{2bt} - 1}{2b}
= e^{2bt}\left(\frac{a}{{2b^2}} + \frac1b\left(r_0 - \frac{a}{b}\right) e^{-bt}
-\frac1b\left(r_0 - \frac{a}{2b}\right) e^{-2bt}\right).
\end{align*}
Consequently, $\E Z_t^2 \le Ce^{2bt}$
with $C = \frac{a}{{2b^2}} + \frac1b\abs{r_0 - \frac{a}{b}}
+\frac1b\abs{r_0 - \frac{a}{2b}}$.

2. Let us consider $\E Z_t^3$.
By It\^o's formula, from \eqref{equ:Z} we have
\begin{equation}\label{equ:Z3}
Z_t^3 = 3 \int_0^t Z_s e^{2bs} r_s\,ds + 3\int_0^t Z_s^2 e^{bs}\sqrt{r_s}\,dW_s.
\end{equation}
Note that according to \citet[Problem 3.15, p.~306]{KaratzasShreve91} \citep[see also][Thm.~9.3]{MishuraShevchenko2017}, for any $p\ge1$, $\E\left [\sup_{t\in[0,T]}\abs{r_t}^{2p}\right ]<\infty$.
It follows from \eqref{equ:CIR-2} and \eqref{equ:Z} that
\begin{equation}\label{equ:r-via-Z}
r_s = e^{-bs}\left(r_0 - \frac{a}{b} + \frac{a}{b}e^{bs} + \sigma Z_s\right)
\end{equation}
Therefore, we have also that $\E\left [\sup_{t\in[0,T]}\abs{Z_t}^{2p}\right ]<\infty$ for any $p\ge1$.
Consequently, all the terms in the equality \eqref{equ:Z3} have bounded expectations.

Hence,
\begin{equation}\label{equ:EZ3}
\E \left[Z_t^3\right] = 3\E \int_0^t Z_s e^{2bs} r_s\,ds.
\end{equation}
We insert \eqref{equ:r-via-Z} into \eqref{equ:EZ3} and obtain
\[
\E \left[Z_t^3\right] = 3\E \int_0^t Z_s e^{bs} \left(r_0 - \frac{a}{b} + \frac{a}{b}e^{bs} + \sigma Z_s\right)\,ds
= 3\sigma \int_0^t e^{bs}  \E\left[Z_s^2\right]\,ds,
\]
since $\E Z_s = 0$.
Applying the bound $\E[Z_s^2] \le C e^{2bs}$ from \eqref{equ:E-bounds}, we get
\[
\E \left[Z_t^3\right] \le 3\sigma C \int_0^t e^{3bs}\,ds
= \frac{C\sigma}{b}\left(e^{3bt}-1\right)
\le \frac{C\sigma}{b} e^{3bt}.
\]

3. We express $r_t$ through $Z_t$ by \eqref{equ:r-via-Z}, and get
\[
\E \left[Z^2_t r_t\right]
=\left(\left(r_0 - \frac{a}{b}\right)e^{-bt} + \frac{a}{b}\right) \E Z_t^2
+ \sigma e^{-bt}\E Z_t^3
\le C e^{2b t},
\]
where the inequality follows from the first two bounds in \eqref{equ:E-bounds}.
\end{proof}

\begin{lem}\label{l:conv-stoch-int}
Let the process $Z$ be defined by \eqref{equ:Z}.
Then the following normalized stochastic integrals vanish as $T\to\infty$:
\begin{gather}
\frac1T\int_0^T e^{-bt} Z_t \sqrt{r_t}\,dW_t \to 0 \quad\text{a.\,s.},
\label{equ:remainder1}
\\
\frac1T e^{-2bT}\int_0^T e^{bt} Z_t \sqrt{r_t}\,dW_t \to 0 \quad\text{a.\,s.}
\label{equ:remainder2}
\end{gather}
\end{lem}

\begin{proof}
1. Obviously, the convergence \eqref{equ:remainder1} is equivalent to
\[
\frac{1}{T+1}\int_0^T e^{-bt} Z_t \sqrt{r_t}\,dW_t \to 0
\quad\text{a.\,s., as } T\to\infty.
\]
By Kronecker's lemma  \citep[see, e.\,g.,][]{Deelstra1995}, in order to prove this convergence it suffices to show that
\begin{equation}\label{equ:char}
\int_0^\infty \frac{e^{-bt} Z_t \sqrt{r_t}}{t+1}\,dW_t <\infty
\quad\text{a.\,s.}
\end{equation}
Since the process $M_T = \int_0^T \frac{e^{-bt} Z_t \sqrt{r_t}}{t+1}\,dW_t$ is a martingale, it suffices to prove that
\begin{equation}\label{equ:char1}
\E \langle M\rangle_T = \E \left[\int_0^T \frac{e^{-2bt} Z_t^2 r_t}{(t+1)^2}\,dt\right ] <\infty.
\end{equation}
But \eqref{equ:char1}  follows immediately from the third bound of \eqref{equ:E-bounds}:
\[
\E \left[\int_0^T \frac{e^{-2bt} Z_t^2 r_t}{(t+1)^2}\,dt\right ]
\le C\int_0^T \frac{1}{(t+1)^2}\,dt
= C\left (1-\frac{1}{T+1}\right )\le C.
\]
Thus, \eqref{equ:remainder1} is proved.

2. Similarly, by Kronecker's lemma, the convergence \eqref{equ:remainder2} follows from the existence a.\,s. of the integral
\[
\int_0^\infty \frac{e^{-bt} Z_t \sqrt{r_t}}{(t+1)e^{2bt}}\,dW_t
= \int_0^\infty \frac{e^{-bt} Z_t \sqrt{r_t}}{t+1}\,dW_t.
\]
This integral is finite, because it coincides with the integral \eqref{equ:char} considered in the first part of the proof.
\end{proof}

Now we are ready to prove the announced asymptotic result for $\frac1T\int_0^T r_t^2\,dt$.
\begin{thm}\label{th:int-r2}
The following convergence holds:
\[
\frac1T\int_0^T r_t^2\,dt \to \frac{a^2}{b^2} + \frac{a\sigma^2}{2b^2}
\quad\text{a.\,s., as } T\to\infty.
\]
\end{thm}

\begin{proof}
Using   relation \eqref{equ:CIR-2}, we get the following equalities:
\begin{align}
\frac1T\int_0^T r_t^2\,dt
&= \frac1T\int_0^T e^{-2bt}\left(r_0 - \frac{a}{b} + \frac{a}{b}e^{bt} + \sigma \int_0^t e^{bu} \sqrt{r_u}\,dW_u\right)^2\,dt
\notag\\
&= \frac1T\int_0^T e^{-2bt}\left(r_0 - \frac{a}{b} + \frac{a}{b}e^{bt}\right)^2\,dt
\notag\\
&\quad+\frac{2\sigma}{T}\int_0^T e^{-2bt}\left(r_0 - \frac{a}{b} + \frac{a}{b}e^{bt}\right)\left(\int_0^t e^{bu} \sqrt{r_u}\,dW_u\right)\,dt
\notag\\
&\quad+ \frac{\sigma^2}{T}\int_0^T e^{-2bt}\left(\int_0^t e^{bu} \sqrt{r_u}\,dW_u\right)^2\,dt
\notag\\
&\eqqcolon I_1 + I_2 + I_3.
\label{equ:I-123}
\end{align}

The term $I_1$ is the subject of  straightforward calculations:
\begin{align*}
I_1 &= \frac1T\int_0^T e^{-2bt}\left(r_0 - \frac{a}{b} + \frac{a}{b}e^{bt}\right)^2\,dt
\\
&= \frac1T\int_0^T \left(\left(r_0 - \frac{a}{b}\right)^2e^{-2bt}
+ \frac{2a}{b}\left(r_0 - \frac{a}{b}\right)e^{-bt}
+\frac{a^2}{b^2}\right) dt.
\\
&= \left(r_0 - \frac{a}{b}\right)^2 \frac{1-e^{-2bT}}{2b T}
+ \frac{2a}{b}\left(r_0 - \frac{a}{b}\right)\frac{1-e^{-bT}}{b T}
+\frac{a^2}{b^2}.
\end{align*}
Evidently, the following convergence holds:
\begin{equation}\label{equ:I_1-conv}
I_1 \to \frac{a^2}{b^2}, \quad T\to\infty.
\end{equation}

Let us consider $I_2$. The equality \eqref{equ:CIR-2} implies that
\[
\sigma \int_0^t e^{bu} \sqrt{r_u}\,dW_u
= e^{bt}r_t-r_0 + \frac{a}{b} - \frac{a}{b}e^{bt}.
\]
Consequently, integral $I_2$ can be transformed as follows:
\begin{align}
I_2 &= \frac{2\sigma}{T}\int_0^T e^{-2bt}\left(r_0 - \frac{a}{b} + \frac{a}{b}e^{bt}\right)\left(\int_0^t e^{bu} \sqrt{r_u}\,dW_u\right)\,dt
\notag\\
&= \frac{2}{T}\int_0^T e^{-2bt}\left(r_0 - \frac{a}{b} + \frac{a}{b}e^{bt}\right)\left(e^{bt}r_t-r_0 + \frac{a}{b} - \frac{a}{b}e^{bt}\right )\,dt
\notag\\
&= \frac{2}{T}\left(r_0 - \frac{a}{b}\right)\int_0^T e^{-bt}r_t\,dt
+ \frac{2a}{bT}\int_0^T r_t\,dt
- \frac{2}{T}\int_0^T e^{-2bt}\left(r_0 - \frac{a}{b} + \frac{a}{b}e^{bt}\right)^2\,dt.
\label{equ:I_2}
\end{align}
Asymptotic relation  \eqref{equ:int-e-r} implies that the first term in the right-hand side of \eqref{equ:I_2} converges to zero a.\,s.\ as $T\to\infty$.
By \eqref{ergo-1}, for the second term we have
\[
\frac{2a}{bT}\int_0^T r_t\,dt \to \frac{2a^2}{b^2} \quad\text{a.\,s., as } T\to\infty.
\]
Finally, the third term equals $-2I_1$, therefore, by \eqref{equ:I_1-conv},
\[
- \frac{2}{T}\int_0^T e^{-2bt}\left(r_0 - \frac{a}{b} + \frac{a}{b}e^{bt}\right)^2\,dt = -2I_1 \to -\frac{2a^2}{b^2}, \quad\text{as } T\to\infty.
\]
Thus,
\begin{equation}\label{equ:I_2-conv}
I_2 \to 0  \quad\text{a.\,s., as } T\to\infty.
\end{equation}

It remains to study the asymptotic behavior of $I_3$.
Note that
\[
I_3 = \frac{\sigma^2}{T}\int_0^T e^{-2bt}Z_t^2\,dt,
\]
where $Z_t$ is defined in \eqref{equ:Z}.
By It\^o's formula, from \eqref{equ:Z} we get
\[
Z_t^2 = \int_0^t e^{2bs} r_s\,ds + 2\int_0^t Z_s e^{bs}\sqrt{r_s}\,dW_s.
\]
Therefore we can present $I_3 $  as the sum
\begin{equation}\label{equ:I_3}
I_3 = \frac{\sigma^2}{T}\int_0^T e^{-2bt}\int_0^t e^{2bs} r_s\,ds\,dt
+\frac{2\sigma^2}{T}\int_0^T e^{-2bt}\int_0^t Z_s e^{bs}\sqrt{r_s}\,dW_s\,dt
\eqqcolon I_{3,1} + I_{3,2}.
\end{equation}
Using the Fubini theorem, we transform $I_{3,1}$ as follows:
\begin{align*}
I_{3,1} &= \frac{\sigma^2}{T}\int_0^T e^{2bs}r_s\int_s^T e^{-2bt}\,dt\,ds
= \frac{\sigma^2}{T}\int_0^T e^{2bs}r_s\frac{e^{-2bs}-e^{-2bT}}{2b}\,ds
\\
&= \frac{\sigma^2}{2bT}\int_0^T r_s\,ds
- \frac{\sigma^2}{2bT}e^{-2bT}\int_0^T e^{2bs}r_s\,ds.
\end{align*}
Using   asymptotic relations   \eqref{ergo-1} and \eqref{equ:e-int-e-r}, we obtain
\begin{equation}\label{equ:I31-conv}
I_{3,1} \to \frac{a\sigma^2}{2b^2},
\quad\text{a.\,s., as } T\to\infty.
\end{equation}
Changing the order of integration, we rewrite $I_{3,2}$ in the following way:
\begin{align*}
I_{3,2} &= \frac{2\sigma^2}{T}\int_0^T  Z_s e^{bs}\sqrt{r_s} \int_s^T e^{-2bt}\,dt\,dW_s
= \frac{2\sigma^2}{T}\int_0^T  Z_s e^{bs}\sqrt{r_s}\, \frac{e^{-2bs} - e^{-2bT}}{2b}\,dt\,dW_s
\\
&= \frac{\sigma^2}{bT}\int_0^T  e^{-bs} Z_s \sqrt{r_s}\,dt\,dW_s
- \frac{\sigma^2}{bT} e^{-2bT}\int_0^T e^{bs} Z_s \sqrt{r_s}\,dt\,dW_s.
\end{align*}
According to Lemma~\ref{l:conv-stoch-int}, both stochastic integrals in the right-hand side converge to zero a.\,s. Therefore, $I_{3,2}\to0$ a.\,s., as $T\to\infty$. Combining this with \eqref{equ:I_3} and \eqref{equ:I31-conv}, we see that
\begin{equation}\label{equ:I_3-conv}
I_3 \to \frac{a\sigma^2}{2b^2}  \quad\text{a.\,s., as } T\to\infty.
\end{equation}

Finally, inserting the convergences \eqref{equ:I_1-conv}, \eqref{equ:I_2-conv}, \eqref{equ:I_3-conv} into the equality \eqref{equ:I-123}, we conclude the proof.
\end{proof}

Theorem~\ref{th:int-r2} enables to construct a strongly consistent estimator for the parameter $(a,b)$.

\begin{thm}\label{th:consist-new}
Define
\begin{align*}
\tilde a_T &= \frac{\sigma^2}{2}\cdot\frac{\left(\int_0^T r_t\,dt\right)^2}{T \int_0^T r_t^2 \,dt - \left(\int_0^T r_t\,dt\right)^2},
\\
\tilde b_T &= \frac{\sigma^2}{2}\cdot\frac{T \int_0^T r_t\,dt}{T \int_0^T r_t^2 \,dt - \left(\int_0^T r_t\,dt\right)^2}.
\end{align*}
Then vector  $(\tilde a_T, \tilde b_T)$ is a strongly consistent estimator of $(a,b)$.
\end{thm}

\begin{proof}
It follows from the convergence \eqref{ergo-1} and Theorem~\ref{th:int-r2} that
\[
\tilde b_T = \frac{\sigma^2}{2}\cdot\frac{\frac1T \int_0^T r_t\,dt}{\frac1T \int_0^T r_t^2 \,dt - \left(\frac1T\int_0^T r_t\,dt\right)^2}
\to \frac{\sigma^2}{2}\cdot\frac{\frac{a}{b}}{\frac{a^2}{b^2} + \frac{a\sigma^2}{2b^2} - \frac{a^2}{b^2}}=b,
\quad\text{a.\,s., as } T\to\infty.
\]
Further, from \eqref{ergo-1} we have the convergence
\[
\tilde a_T = \tilde b_T \cdot \frac1T \int_0^T r_t\,dt
\to b\cdot \frac{a}{b}=a,
\quad\text{a.\,s., as } T\to\infty.
\qedhere
\]
\end{proof}

\begin{remark}
For the model \eqref{equ:10} the alternative estimator of $(\alpha,\mu)$ can be defined as follows:
\begin{align*}
\tilde \alpha_T &= \frac{\sigma^2}{2}\cdot\frac{T \int_0^T r_t\,dt}{T \int_0^T r_t^2 \,dt - \left(\int_0^T r_t\,dt\right)^2},
\qquad
\tilde\mu_T = \frac1T \int_0^T r_t\,dt.
\end{align*}
The strong consistency of the estimator $(\tilde \alpha_T, \tilde \mu_T)$ follows from Theorem~\ref{th:consist-new} and the relations~\eqref{equ:dif-param}.
\end{remark}
\vskip 3mm
    	
\noindent 5.    SIMULATIONS

In this section we illustrate the quality of the estimators, using Monte Carlo simulation.
For each set of parameters $(a,b,\sigma)$, we generate 100 sample paths of the solution $r=\set{r_t,t\in[0,T]}$ to the equation \eqref{equ:1} using Euler's approximation.
We choose the initial value $r_0=1$ for all simulations.
Then we compute means and standard deviations of the estimators at the times $T=10,50,100,150,200$.

In the ergodic case $2a>\sigma^2$ we compare two estimators, $(\hat a_T, \hat b_T)$ and $(\tilde a_T, \tilde b_T)$.
Tables~\ref{tab:1} and \ref{tab:2} report the simulation results concerning the estimation of $a$ and $b$ respectively.
We see that both estimators are consistent and behave similarly, however, the maximum likelihood estimator $(\hat a_T, \hat b_T)$ slightly outperforms the alternative estimator $(\tilde a_T, \tilde b_T)$.

For the non-ergodic case $2a<\sigma^2$ the maximum likelihood estimator does not make sense, so only the means and variances for $(\tilde a_T, \tilde b_T)$ are reported, see Tables~\ref{tab:3} and \ref{tab:4}. Clearly, the numerical results confirm the consistency of this estimator. However, in many cases the rate of convergence is quite slow (compared to the ergodic case), especially for $\tilde b_T$.

\renewcommand{\baselinestretch}{1}
\begin{table}
\centering
\caption{Means and standard deviations of $\hat a_T$ and $\tilde a_T$ for $2a>\sigma^2$\label{tab:1}}
\footnotesize
\begin{tabular}{*{9}{Q}}\toprule
&&&&& T &&&\\
\cmidrule(l){5-9}
a & b & \sigma &  & 10 & 50 & 100 & 150 & 200\\\otoprule
1&1&1& \E [\hat a_T]  & 1.3520 & 1.0507 & 1.0187 & 1.0105 & 1.0073\\
&&&    \sd[\hat a_T] &  0.5357 & 0.1628 & 0.0983 & 0.0818 & 0.0723\\
\addlinespace
&&& \E [\tilde a_T]  & 1.5328 & 1.1152 & 1.0535 & 1.0300 & 1.0226\\
&&& \sd[\tilde a_T]  & 0.5804 & 0.2435 & 0.1701 & 0.1306 & 0.1203\\
\midrule
1&2&1& \E [\hat a_T]  & 1.2013 & 1.0401 & 1.0144 & 1.0082 & 1.0046\\
&&&    \sd[\hat a_T] &  0.3201 & 0.1077 & 0.0700 & 0.0557 & 0.0495\\
\addlinespace
&&& \E [\tilde a_T]  & 1.3028 & 1.0766 & 1.0270 & 1.0206 & 1.0140\\
&&& \sd[\tilde a_T]  & 0.4304 & 0.1618 & 0.1051 & 0.0974 & 0.0872\\
\midrule
1&3&1& \E [\hat a_T]  & 1.0915 & 1.0201 & 1.0192 & 1.0129 & 1.0141\\
&&&    \sd[\hat a_T] &  0.2029 & 0.0821 & 0.0559 & 0.0470 & 0.0459\\
\addlinespace
&&& \E [\tilde a_T]  & 1.1142 & 1.0411 & 1.0241 & 1.0252 & 1.0239\\
&&& \sd[\tilde a_T]  & 0.2961 & 0.1282 & 0.0944 & 0.0810 & 0.0778\\
\midrule
2&1&1& \E [\hat a_T]  & 2.8227 & 2.0724 & 1.9885 & 1.9909 & 1.9963\\
&&&    \sd[\hat a_T] &  1.1649 & 0.3163 & 0.2316 & 0.2087 & 0.1908\\
\addlinespace
&&& \E [\tilde a_T]  & 2.8584 & 2.1097 & 2.0035 & 2.0030 & 2.0157\\
&&& \sd[\tilde a_T]  & 1.1089 & 0.4139 & 0.3142 & 0.2724 & 0.2446\\
\midrule
2&2&1& \E [\hat a_T]  & 2.2837 & 2.0511 & 2.0353 & 2.0170 & 2.0216\\
&&&    \sd[\hat a_T] &  0.6633 & 0.2375 & 0.1662 & 0.1365 & 0.1173\\
\addlinespace
&&& \E [\tilde a_T]  & 2.4734 & 2.1057 & 2.0447 & 2.0057 & 2.0065\\
&&& \sd[\tilde a_T]  & 0.7880 & 0.2990 & 0.2134 & 0.1784 & 0.1560\\
\midrule
2&3&1& \E [\hat a_T]  & 2.2114 & 2.0113 & 2.0120 & 2.0098 & 2.0091\\
&&&    \sd[\hat a_T] &  0.5337 & 0.1986 & 0.1509 & 0.1057 & 0.0915\\
\addlinespace
&&& \E [\tilde a_T]  & 2.2874 & 2.0208 & 2.0224 & 2.0195 & 2.0103\\
&&& \sd[\tilde a_T]  & 0.5829 & 0.2602 & 0.1825 & 0.1379 & 0.1217\\
\midrule
3&1&1& \E [\hat a_T]  & 3.9869 & 3.1923 & 3.0761 & 3.0475 & 3.0538\\
&&&    \sd[\hat a_T] & 1.2874 & 0.5061 & 0.4039 & 0.3421 & 0.2991\\
\addlinespace
&&& \E [\tilde a_T]  & 3.7772 & 3.1524 & 3.0538 & 3.0347 & 3.0475\\
&&& \sd[\tilde a_T]  & 1.3190 & 0.5808 & 0.4591 & 0.4042 & 0.3519\\
\midrule
3&1&2& \E [\hat a_T]  & 3.7755 & 3.1447 & 3.0612 & 3.0358 & 3.0359\\
&&&    \sd[\hat a_T] &  1.2919 & 0.3776 & 0.2618 & 0.2124 & 0.1919\\
\addlinespace
&&& \E [\tilde a_T]  & 4.1227 & 3.3441 &  3.1022 & 3.0743 & 3.0872\\
&&& \sd[\tilde a_T]  & 1.5027 & 0.6588 & 0.5050 & 0.4323 & 0.4014\\
\bottomrule
\end{tabular}
\end{table}

\begin{table}
\centering
\caption{Means and standard deviations of $\hat b_T$ and $\tilde b_T$ for $2a>\sigma^2$ \label{tab:2}}
\footnotesize
\begin{tabular}{*{9}{Q}}\toprule
&&&&& T &&&\\
\cmidrule(l){5-9}
a & b & \sigma &  & 10 & 50 & 100 & 150 & 200\\\otoprule
1&1&1& \E [\hat b_T]  & 1.4514 & 1.0772 & 1.0244 & 1.0176 & 1.0132\\
&&&    \sd[\hat b_T] &  0.5969 & 0.2184 & 0.1398 & 0.1086 & 0.0986\\
\addlinespace
&&& \E [\tilde b_T]  & 1.6350 & 1.1459 & 1.0554 & 1.0363 & 1.0271\\
&&& \sd[\tilde b_T]  & 0.6687 & 0.2983 & 0.1829 & 0.1437 & 0.1293\\
\midrule
1&2&1& \E [\hat b_T]  & 2.4399 & 2.1229 & 2.0677 & 2.0335 & 2.0159\\
&&&    \sd[\hat b_T] &  0.6993 & 0.2658 & 0.1941 & 0.1693 & 0.1587\\
\addlinespace
&&& \E [\tilde b_T]  & 2.5135 & 2.1743 & 2.0851 & 2.0526 & 2.0303\\
&&& \sd[\tilde b_T]  & 0.7686 & 0.3495 & 0.2628 & 0.2393 & 0.2178\\
\midrule
1&3&1& \E [\hat b_T]  & 3.3578 & 3.0316 & 3.0287 & 3.0094 & 3.0162\\
&&&    \sd[\hat b_T] &  0.8546 & 0.3701 & 0.2479 & 0.2071 & 0.1949\\
\addlinespace
&&& \E [\tilde b_T]  & 3.1754 & 3.0495 &  3.0243 & 3.0313 & 3.0339\\
&&& \sd[\tilde b_T]  & 0.8816 & 0.4501 & 0.3351 & 0.2731 & 0.2578\\
\midrule
2&1&1& \E [\hat b_T]  & 1.4767 & 1.0521 &  0.9980 & 0.9967 & 0.9998\\
&&&    \sd[\hat b_T] &  0.6254 & 0.1944 & 0.1323 & 0.1214 & 0.1101\\
\addlinespace
&&& \E [\tilde b_T]  & 1.5466 & 1.0813 & 1.0102 & 1.0056 & 1.0120\\
&&& \sd[\tilde b_T]  & 0.5977 & 0.2338 & 0.1702 & 0.1475 & 0.1333\\
\midrule
2&2&1& \E [\hat b_T]  & 2.3087 & 2.0703 & 2.0437 & 2.0210 & 2.0216\\
&&&    \sd[\hat b_T] &  0.6851 & 0.2865 & 0.1996 & 0.1600 & 0.1324\\
\addlinespace
&&& \E [\tilde b_T]  & 2.4894 & 2.1255 & 2.0536 & 2.0104 & 2.0074\\
&&& \sd[\tilde b_T]  & 0.7913 & 0.3422 & 0.2441 & 0.2007 & 0.1755\\
\midrule
2&3&1& \E [\hat b_T]  & 3.3561 & 3.0417 & 3.0420 & 3.0354 & 3.0311\\
&&&    \sd[\hat b_T] &  0.8854 & 0.3546 & 0.2705 & 0.1811 & 0.1619\\
\addlinespace
&&& \E [\tilde b_T]  & 3.4176 & 3.0446 & 3.0535 & 3.0465 & 3.0305\\
&&& \sd[\tilde b_T]  & 0.9314 & 0.4384 & 0.3185 & 0.2254 & 0.1988\\
\midrule
3&1&1& \E [\hat b_T]  & 1.3899 & 1.0798 & 1.0337 & 1.0264 & 1.0262\\
&&&    \sd[\hat b_T] & 0.5563 & 0.1882 & 0.1525 & 0.1283 & 0.1109\\
\addlinespace
&&& \E [\tilde b_T]  & 1.3878 & 1.0797 & 1.0326 & 1.0264 & 1.0268\\
&&& \sd[\tilde b_T]  & 0.5825 & 0.2175 &  0.1688 & 0.1451 & 0.1250\\
\midrule
3&1&2& \E [\hat b_T]  & 1.3616 & 1.0617 & 1.0209 & 1.0038 & 0.9989\\
&&&    \sd[\hat b_T] & 0.4884 & 0.2102 & 0.1361 & 0.1011 & 0.0977\\
\addlinespace
&&& \E [\tilde b_T] & 1.6055 & 1.1483 & 1.0447 & 1.0220 & 1.0185\\
&&& \sd[\tilde b_T]  & 0.6558 & 0.3057 & 0.2179 & 0.1673 & 0.1519\\
\bottomrule
\end{tabular}
\end{table}

\begin{table}
\centering
\caption{Means and standard deviations $\tilde a_T$ for $2a<\sigma^2$\label{tab:3}}
\footnotesize
\begin{tabular}{*{9}{Q}}\toprule
&&&&& T &&&\\
\cmidrule(l){5-9}
a & b & \sigma &  & 10 & 50 & 100 & 150 & 200\\\otoprule
1&1&2& \E [\tilde a_T]  & 1.6220  & 1.1125 & 1.0583 & 1.0384 & 1.0348\\
&&& \sd[\tilde a_T]  & 0.6134 & 0.2720 & 0.2164 & 0.1826 & 0.1589\\
\midrule
1&1&3& \E [\tilde a_T]  & 1.3276 & 1.1253 & 1.1256 & 1.1407 & 1.1091\\
&&& \sd[\tilde a_T]  & 2.5786 & 1.0947 & 0.6717 & 0.5451 & 0.4372\\
\midrule
1&2&2& \E [\tilde a_T]  & 1.3309 & 1.0907 & 1.0570 & 1.0357 & 1.0212\\
&&& \sd[\tilde a_T]  & 0.3762 & 0.2141 & 0.1442 & 0.1317 & 0.1144\\
\midrule
1&2&3& \E [\tilde a_T]  & 1.2227 & 1.1028 & 1.1034 & 1.0695 & 1.0367\\
&&& \sd[\tilde a_T]  & 2.8947 & 0.6768 & 0.4528 & 0.3581 & 0.2840\\
\midrule
1&3&2& \E [\tilde a_T]  & 1.3113 & 1.1099 & 1.0618 & 1.0432 & 1.0330\\
&&& \sd[\tilde a_T]  & 0.3936 & 0.1917 & 0.1263 & 0.1137 & 0.1020\\
\midrule
1&3&3& \E [\tilde a_T]  & 1.3829 & 1.2046 & 1.0947 & 1.0544 & 1.0404\\
&&& \sd[\tilde a_T]  & 2.2021 & 0.8893 & 0.3813 & 0.2773 & 0.2507\\
\midrule
2&1&3& \E [\tilde a_T]  & 3.3746 & 2.3050 & 2.1699 & 2.1496 & 2.1363\\
&&& \sd[\tilde a_T]  & 1.5845 & 0.5720 & 0.4390 & 0.3586 & 0.3095\\
\midrule
2&2&3& \E [\tilde a_T]  & 2.7734 & 2.2172 & 2.1665 & 2.1346 & 2.1137\\
&&& \sd[\tilde a_T]  & 0.9826 & 0.4067 & 0.3020 & 0.2526 & 0.2273\\
\midrule
2&3&3& \E [\tilde a_T]  & 2.5827 & 2.1628 & 2.1024 & 2.0647 & 2.0350\\
&&& \sd[\tilde a_T]  & 0.7142 & 0.3657 & 0.2816 & 0.2316 & 0.1994\\
\midrule
3&1&3& \E [\tilde a_T]  & 4.5037 & 3.4455 & 3.2534 & 3.1348 & 3.0772\\
&&& \sd[\tilde a_T]  & 1.6354 & 0.7223 & 0.5435 & 0.5096 & 0.4290\\
\bottomrule
\end{tabular}
\end{table}

\begin{table}
\centering
\caption{Means and standard deviations $\tilde b_T$ for $2a<\sigma^2$\label{tab:4}}
\footnotesize
\begin{tabular}{*{9}{Q}}\toprule
&&&&& T &&&\\
\cmidrule(l){5-9}
a & b & \sigma &  & 10 & 50 & 100 & 150 & 200\\\otoprule
1&1&2& \E [\tilde b_T]  & 2.0929 & 1.2595 & 1.1232 & 1.0907 & 1.0621\\
&&& \sd[\tilde b_T]  & 1.2187 & 0.4458 & 0.3275 & 0.2782 & 0.2326\\
\midrule
1&1&3& \E [\tilde b_T]  & 1.4055 & 1.0156 & 1.1587 & 1.2190 & 1.1632\\
&&& \sd[\tilde b_T]  & 2.6746 & 1.9533 & 0.9132 & 0.7432 & 0.6196\\
\midrule
1&2&2& \E [\tilde b_T]  & 2.9370 & 2.2790 & 2.1923 & 2.1499 & 2.1059\\
&&& \sd[\tilde b_T]  & 1.2131 & 0.5691 & 0.3801 & 0.3290 & 0.2883\\
\midrule
1&2&3& \E [\tilde b_T]  & 1.7045 & 2.2836 & 2.3097 & 2.2555 & 2.1756\\
&&& \sd[\tilde b_T]  & 9.6195 & 2.1124 & 1.1064 & 0.8718 & 0.6762\\
\midrule
1&3&2& \E [\tilde b_T]  & 3.9779 & 3.3292 & 3.2273 & 3.1416 & 3.1023\\
&&& \sd[\tilde b_T]  & 1.5206 & 0.7385 & 0.5665 & 0.5010 & 0.4566\\
\midrule
1&3&3& \E [\tilde b_T]  & 3.8650 & 3.8265 & 3.4710 & 3.2583 & 3.1462\\
&&& \sd[\tilde b_T]  & 13.9606 & 4.6146 & 1.5811 & 0.9728 & 0.8789\\
\midrule
2&1&3& \E [\tilde b_T]  & 2.2511 & 1.2859 & 1.1613 & 1.1223 & 1.1047\\
&&& \sd[\tilde b_T]  & 1.6015 & 0.4753 & 0.3403 & 0.2929 & 0.2429\\
\midrule
2&2&3& \E [\tilde b_T]  & 3.2242 & 2.2986 & 2.2252 & 2.1454 & 2.1101\\
&&& \sd[\tilde b_T]  & 1.5489 & 0.5355 & 0.3979 & 0.3355 & 0.2954\\
\midrule
2&3&3& \E [\tilde b_T]  & 4.1619 & 3.3906 & 3.2490 & 3.1841 & 3.1132\\
&&& \sd[\tilde b_T]  & 1.7422 & 0.7886 & 0.5821 & 0.4935 & 0.4223\\
\midrule
3&1&3& \E [\tilde b_T]  & 2.0684 & 1.3138 & 1.1622 & 1.1039 & 1.0754\\
&&& \sd[\tilde b_T]  & 0.9192 & 0.3868 & 0.2481 & 0.2418 & 0.2081\\
\bottomrule
\end{tabular}
\end{table}

\renewcommand{\baselinestretch}{1.5}
\vskip 3mm

\renewcommand{\refname}

\end{document}